\documentclass[graybox]{svmult}

\usepackage{mathptmx}       
\usepackage{helvet}         
\usepackage{courier}        
\usepackage{type1cm}        
\usepackage{makeidx}         
\usepackage{graphicx}        
\usepackage{multicol}        
\usepackage[bottom]{footmisc}

\usepackage{amssymb}
\usepackage{tikz}
\usetikzlibrary{arrows,shapes}
\usepackage{amsmath}

\def\E{\mathbf{E}}
\def\H{\mathbf{H}}
\def\e{\mathbf{e}}
\def\h{\mathbf{h}}
\def\v{\mathbf{v}}
\def\q{\mathbf{q}}
\def\a{\mathbf{a}}
\def\b{\mathbf{b}}
\def\n{\mathbf{n}}
\def\eqref#1{(\ref{#1})}
\def\M{\mathrm{M}}
\def\C{\mathrm{C}}
\def\P{\mathrm{P}}
\def\K{\mathrm{K}}
\def\T{\mathcal{T}}
\def\RR{\mathbb{R}}
\def\dt{\partial_t}
\def\dtt{\partial_{tt}}
\def\curl{\mathop{\mbox{curl}}}

\def\tnorm{|\!|\!|}

\begin{document}

\title*{A mass-lumped mixed finite element method for Maxwell's equations}
\author{Herbert Egger \and Bogdan Radu}
\institute{H. Egger \at Dept. of Mathematics, TU Darmstadt, Dolivostra{\ss}e 15, 64293 Darmstadt, Germany \email{egger@mathematik.tu-darmstadt.de}
\and B. Radu \at Graduate School for Computational Engineering, TU Darmstadt, Dolivostra{\ss}e 15, 64293 Darmstadt, Germany \email{radu@gsc.tu-darmstadt.de}}
\maketitle

\abstract{A novel mass-lumping strategy for a mixed finite element approximation of Maxwell's equations is proposed. 
On structured orthogonal grids the resulting method coincides with the spatial discretization of the Yee scheme. 
The proposed method, however, generalizes naturally to unstructured grids and anisotropic materials and thus yields a variational extension of the Yee scheme for these situations. }

\section{Introduction}
\label{sec:intro}

We consider the propagation of electromagnetic radiation through a linear non-dispersive and non-conducting medium  described by Maxwell’s equations
\begin{eqnarray} 
\epsilon \dt \E &=& \curl \H, \label{eq:maxwell1}\\
\mu \dt \H &=& - \curl \E.    \label{eq:maxwell2}
\end{eqnarray}
Here $\E$, $\H$ denote the electric and magnetic field intensities and 
$\epsilon$, $\mu$ are the symmetric and positive definite permittivity and permeability tensors.
For ease of notation, we assume that $\E \times \n = 0$ on the boundary. 
The space discretization of \eqref{eq:maxwell1}--\eqref{eq:maxwell2} 
by standard methods leads to finite dimensional differential equations of the form
\begin{eqnarray}
\M_\epsilon \dt \e &=& \C' \h, \label{eq:sys1}\\
\M_\mu \dt \h &=&  -\C \e.               \label{eq:sys2}
\end{eqnarray}
Due to the particular structure of the system, the stability of such discretization schemes can easily be
ensured by the simple algebraic conditions 
\begin{itemize}\setlength\itemindent{1em}
 \item[(i)] $\quad \C' = \C^\top$,
 \item[(ii)] $\quad \M_\epsilon$, $\M_\mu$ are symmetric and positive definite.  
\end{itemize}
In order to enable an efficient solution of \eqref{eq:sys1}--\eqref{eq:sys2} by 
explicit time-stepping methods, one additionally has to assume that 
\begin{itemize}\setlength\itemindent{1em}
 \item[(iii)] $\quad \M_\epsilon^{-1}$, $\M_\mu^{-1}$ can be applied efficiently.
\end{itemize}
The finite difference approximation of \eqref{eq:maxwell1}--\eqref{eq:maxwell2} on staggered orthogonal grids yields approximations of the form \eqref{eq:sys1}--\eqref{eq:sys2} satisfying (i)--(iii) with diagonal $\M_\epsilon$, $\M_\mu$  \cite{Yee66}. 
Moreover, the entries $\e_i$, $\h_j$ in the solution vectors yield second order approximations for the line integrals of $\E$, $\H$ along edges of the primal and duals grid, respectively \cite{Cohen02,Weiland96}.
An extension to unstructured grids and anisotropic coefficients is possible  \cite{CodecasaPoliti08,SchuhmannWeiland98}, but these approaches rely on the use of two sets of grids which makes a rigorous convergence analysis difficult.

The finite element approximation of \eqref{eq:maxwell1}--\eqref{eq:maxwell2} yields systems of the form \eqref{eq:sys1}--\eqref{eq:sys2} satisfying (i)--(ii) automatically and a rigorous convergence analysis is possible 
in rather gerneral situations \cite{Joly03,Monk92}.
Condition (iii) is, however, not valid in general, although the matrices $\M_\epsilon$ and $\M_\mu$ are usually sparse.
This lack of efficiency can be overcome by mass-lumping, which aims at approximating $\M_\epsilon$ and $\M_\mu$ by diagonal or block-diagonal matrices; \cite{ElmkiesJoly97,CohenMonk98}. These approaches are usually based on an enrichment of the approximation spaces and appropriate quadrature; see \cite{Cohen02}.

In this paper, we present a novel mass-lumping strategy for a mixed finite element approximation of \eqref{eq:maxwell1}--\eqref{eq:maxwell2} that yields properties (i)--(iii) without such an increase of the system dimension. In special cases, the resulting scheme reduced to the staggered-grid finite difference approximation of the Yee scheme.

\section{A mass-lumped mixed finite element method}
\label{sec:2}

As a preliminary step, we consider a mass-lumped mixed finite element approximation 
based on enriched approximation spaces and numerical quadrature. 
We seek for approximations $\widetilde \E_h(t) \in \widetilde V_h$, $\widetilde \H_h(t) \in \widetilde Q_h$ satisfying
\begin{alignat}{2}
(\epsilon \dt \widetilde \E_h(t),\widetilde \v_h)_h &= (\widetilde \H_h(t), \curl \widetilde \v_h) 
\qquad &&\forall \widetilde \v_h \in \widetilde V_h,\label{eq:var1}\\
(\mu \dt \widetilde \H_h(t),\widetilde \q_h)_{h,*} &= -(\curl \widetilde \E_h(t),\widetilde \q_h)
\qquad &&\forall \widetilde \q_h \in \widetilde Q_h, \label{eq:var2}
\end{alignat}
for all $t>0$. 
Here, $\widetilde V_h \subset H_0(\curl;\Omega)$ and $\widetilde Q_h \subset L^2(\Omega)$ are appropriate finite dimensional subspaces
and $(\a,\b)_h$, $(\a,\b)_{h,*}$ are approximations for the scalar product $(\a,\b)=\int_\Omega \a(x) \cdot \b(x) \; dx$
to be defined below.

We restrict our discussion in the sequel to problems where $\E = (E_x,E_y,0)$ and $\H=(0,0,H_z)$ with $E_x,E_y,H_z$ independent of $z$, which allows to represent the fields in two dimensions. The extension to three dimensions will be discussed in Section~\ref{sec:5}. 

Let $\T_h = \{T\}$ be a conforming mesh of $\Omega$ consisting of triangles and parallelograms. Every element $T \in \T_h$ is the image $F_T(\widehat T)$ of a reference triangle or reference square under an affine mapping $F_T(\widehat x) = a_T + B_T \widehat x$ with $a_T \in \RR^2$ and $B_T \in \RR^{2 \times 2}$. 
We denote by $h$ the maximal element diameter and assume uniform shape regularity.

To every element $T_j$, $j=1,\ldots,n_T$ of the mesh, we associate a basis function $\widetilde \psi_j$ 
for the space $\widetilde Q_h$ with $\widetilde \psi_j|_{T_k} = \delta_{jk}$. 
\begin{figure}[ht!]
\centering
\begin{tikzpicture}[scale=0.55]
\draw (0,0) -- (0,3) -- (3,3) -- (3,0) -- (0,0);
\draw[thick,->] (0.2,-0.3) -- (1.2,-0.3);
\draw[thick,->] (1.8,-0.3) -- (2.8,-0.3);
\draw[thick,->] (1.2,3.3) -- (0.2,3.3);
\draw[thick,->] (2.8,3.3) -- (1.8,3.3);
\draw[thick,->] (-0.3,1.2) -- (-0.3,0.2);
\draw[thick,->] (-0.3,2.8) -- (-0.3,1.8);
\draw[thick,->] (3.3,0.2) -- (3.3,1.2);
\draw[thick,->] (3.3,1.8) -- (3.3,2.8);
\draw[thick] (1.5,1.5) circle (0.2cm);
\draw[fill] (3,0) circle (0.12cm);
\draw[fill] (0,3) circle (0.12cm);
\draw[fill] (3,3) circle (0.12cm);
\draw[fill] (0,0) circle (0.12cm);
\end{tikzpicture}
\hspace*{13em}
\begin{tikzpicture}[scale=0.6]
\draw (0,0) -- (0,3) -- (3,0) -- (0,0);
\draw[thick,->] (0.2,-0.3) -- (1.2,-0.3);
\draw[thick,->] (1.8,-0.3) -- (2.8,-0.3);
\draw[thick,->] (-0.3,1.2) -- (-0.3,0.2);
\draw[thick,->] (-0.3,2.8) -- (-0.3,1.8);
\draw[thick,->] (3.0,0.5) -- (2.2,1.3);
\draw[thick,->] (1.3,2.2) -- (0.5,3.0);

\draw[thick] (1,1) circle (0.15cm);

\draw[fill] (3,0) circle (0.12cm);
\draw[fill] (0,3) circle (0.12cm);
\draw[fill] (0,0) circle (0.12cm);
\end{tikzpicture}\\
\vspace*{1em}
\hspace*{2em}
\begin{minipage}[t]{.48\textwidth}
$\widehat\phi_{1,0}=\frac{1}{2}\binom{-y^2+y}{-2xy+2x}$,\quad $\widehat\phi_{1,1}=\frac{1}{2}\binom{y^2-y}{2xy}$,\\
$\widehat\phi_{2,0}=\frac{1}{2}\binom{-2xy}{-x^2+x}$,\quad $\widehat\phi_{2,1}=\frac{1}{2}\binom{-2y+2xy}{x^2-x}$,\\
$\widehat\phi_{3,0}=\frac{1}{2}\binom{y^2-y}{2xy-2y}$,\quad $\widehat\phi_{3,1}=\frac{1}{2}\binom{-y^2+y}{2x+2y-2xy-2}$,\\
$\widehat\phi_{4,0}=\frac{1}{2}\binom{-2x-2y+2xy+2}{x^2-x}$,\quad $\widehat\phi_{4,1}=\frac{1}{2}\binom{-2xy+2x}{-x^2+x}$.
\end{minipage}
\hspace*{2.5em}
\begin{minipage}[t]{.35\textwidth}
$\widehat\phi_{1,0}=\frac{1}{2}\binom{0}{x}$,\quad $\widehat\phi_{1,1}=\frac{1}{2}\binom{-y}{0}$,\\
$\widehat\phi_{2,0}=\frac{1}{2}\binom{-y}{-y}$,\quad $\widehat\phi_{2,1}=\frac{1}{2}\binom{0}{x+y-1}$,\\
$\widehat\phi_{3,0}=\frac{1}{2}\binom{1-x-y}{0}$,\quad $\widehat\phi_{3,1}=\frac{1}{2}\binom{x}{x}$.
\end{minipage}
\caption{Degrees of freedom and basis functions for the unit triangle and unit square. The black dots at the vertices represent the quadrature points for the quadrature formula introduced below. \label{fig:basis}}
\end{figure}
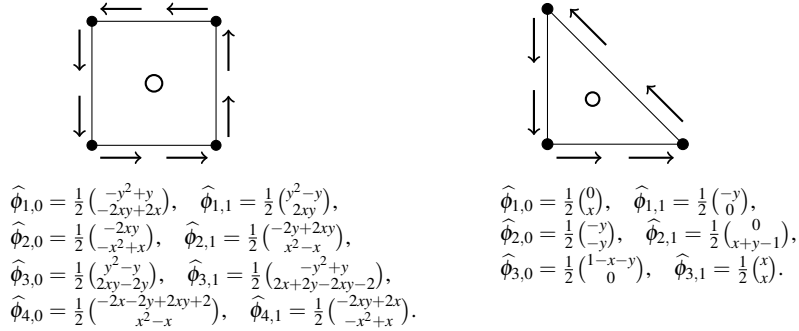
For every interior edge $e_i=T_l \cap T_r$, $i=1,\ldots,n_e$ of the mesh, we further define two basis functions $\widetilde \phi_i,\widetilde \phi_{i+n_e}$ which are defined by
\begin{equation} \label{eq:piola}
\widetilde \phi_{i+\ell\cdot n_e}|_T = B_T^{-\top} \widehat \phi_{\alpha,\gamma}, \qquad \ell=0,1,
\end{equation}
on $T \in \{T_l,T_r\}$ and vanish identically on all other elements. 
Here $\alpha \in \{1,\ldots,\widehat n_e\}$ refers to the number of the edge $e_i$ on the reference element $\widehat T$ and $\gamma \in \{0,1\}$ depends on $\ell$ and the orientation of the edge $e_i$. The functions $\widehat \phi_{\alpha,\gamma}$ are defined in Figure~\ref{fig:basis}. 
We further set $(\a,\b)_{h,*}=(\a,\b)$ and define $(\a,\b)_h = \sum\nolimits_{T} (\a,\b)_{h,T}$ with 
\begin{equation}
(\a,\b)_{h,T}=|T| \sum\nolimits_{l=1}^{\widehat n_p} \a(F_T(\widehat x_l)) \cdot \b(F_T(\widehat x_l)) \; w_l,
\end{equation}
where $\widehat x_l$, $l=1,\ldots,\widehat n_p$ denote the quadrature points and $w_l=1/\widehat n_p$ the quadrature weights on the reference element as depicted in Figure~\ref{fig:basis}.

Using the bases defined above, all functions in $\widetilde V_h$ and $\widetilde Q_h$ can be represented as 
\begin{align}
\widetilde \E_h = \sum\nolimits_i \widetilde \e_i \widetilde \phi_i + \widetilde \e_{i+n_e} \widetilde \phi_{i+n_e} 
\qquad \text{and} \qquad 
\widetilde \H_h = \sum\nolimits_j \widetilde \h_j \widetilde\psi_j.
\end{align}
This allows to rewrite the variational problem \eqref{eq:var1}--\eqref{eq:var2} in algebraic form as 
\begin{align}
\widetilde \M_\epsilon \dt \widetilde \e = \widetilde \C^\top \widetilde \h \label{eq:sys1w} \\
\widetilde \M_\mu \dt \widetilde \h = - \widetilde \C\,\widetilde \e \label{eq:sys2w}
\end{align}
with matrices $(\widetilde \M_\epsilon)_{ij} = (\epsilon \widetilde \phi_j,\widetilde \phi_i)_h$, $(\widetilde \M_\mu)_{ij} = (\mu \widetilde \psi_j,\widetilde\psi_i)$, and $(\widetilde \C)_{ij}=(\curl \widetilde\phi_j,\widetilde\psi_i)$. 
As a direct consequence of the particular choice of the basis functions, we obtain 
\begin{lemma} \label{lem:1}
Let $\widetilde \M_\epsilon$, $\widetilde \M_\mu$, and $\widetilde \C$ be defined as above. 
Then (i)--(iii) hold analogously.
\end{lemma}
\begin{proof}
The properties (i)--(ii) follow directly form the construction. 
From the particular choice of basis functions, one can deduce that $\widetilde\M_\mu$ is diagonal and $\widetilde\M_\epsilon$ is block-diagonal; see \cite{EggerRadu18b,WheelerYotov06} for details. This implies conditions (iii). \qed
\end{proof}

Let us mention that the quadrature rule satisfies $(\a,\b)_{h,T}=\int_T \a(x) \cdot \b(x) \; dx$ when $\a(x)\cdot \b(x)$ is affine linear. This ensures that the method \eqref{eq:var1}--\eqref{eq:var2} also has good approximation properties. 
By a slight adoption of the results given in \cite{EggerRadu18b}, we obtain
\begin{lemma} \label{lem:2}
Let $\E$, $\H$ be a smooth solution of \eqref{eq:maxwell1}--\eqref{eq:maxwell2} and let $\widetilde \E_h(0)$ and $\widetilde \H_h(0)$ be chosen appropriately. Then 
\begin{align*}  
\|\widetilde \E_h(t) - \E(t)\| + \|\widetilde \H_h(t) - \H(t)\| \le C h
\end{align*}
for all $0 \le t \le T$ with $C=C(E,H,T)$.
Moreover, $\|\widetilde \H_h(t) - \pi_h^0 \H(t)\| \le C h^2$ where $\pi_h^0 \H$ denotes the piecewise constant approximation of $\H$ on the mesh $\T_h$. 
\end{lemma}
\begin{remark}
For structured meshes and isotropic coefficients, one can observe second order convergence also  for line integrals of the electric field along edges of the mesh. In addition, second convergence can also obtained for unstructured meshes by a non-local post-processing strategy; see \cite{EggerRadu18b} for details.
\end{remark} 

\section{A variational extension of the Yee scheme}
\label{sec:3}

The method of the previous section already yields a stable and efficient approximation. 
We now show that one degree of freedom per edge can be saved without sacrificing the accuracy or efficiency of the method.
To this end, we construct approximations $\E_h(t) \in V_h$, $\H_h(t) \in Q_h$ in spaces 
$V_h \subset \widetilde V_h$ and $Q_h = \widetilde Q_h$.

\begin{figure}[ht!]
\centering
\begin{tikzpicture}[scale=0.55]
\draw (0,0) -- (0,3) -- (3,3) -- (3,0) -- (0,0);
\draw[thick,->] (0.5,-0.3) -- (2.5,-0.3);
\draw[thick,->] (2.5,3.3) -- (0.5,3.3) ;
\draw[thick,->] (-0.3,2.5) -- (-0.3,0.5);
\draw[thick,->] (3.3,0.5) -- (3.3,2.5);
\end{tikzpicture} 
\hspace*{13em}
\begin{tikzpicture}[scale=0.6]
\draw (0,0) -- (0,3) -- (3,0) -- (0,0);
\draw[thick,->] (0.5,-0.3) -- (2.5,-0.3);
\draw[thick,->] (-0.3,2.5) -- (-0.3,0.5);
\draw[thick,->] (2.5,1.0) -- (1.0,2.5);
\end{tikzpicture}\\
\vspace*{1em}
\begin{minipage}[t]{.35\textwidth}
$\widehat\phi_{1}=\binom{0}{x}$,\quad $\widehat\phi_{2}=\binom{-y}{0}$,\\
$\widehat\phi_{3}=\binom{0}{x-1}$,\quad $\widehat\phi_{4}=\binom{1-y}{0}$.
\end{minipage}
\hspace*{4em}
\begin{minipage}[t]{.30\textwidth}
$\widehat\phi_{1}=\frac{1}{2}\binom{-y}{x}$,\quad $\widehat\phi_{2}=\frac{1}{2}\binom{-y}{x-1}$\\
$\widehat\phi_{3}=\frac{1}{2}\binom{1-y}{x}$.
\end{minipage}
\caption{Degrees of freedom and basis functions on the unit triangle and unit square. \label{fig:basis2}}
\end{figure}
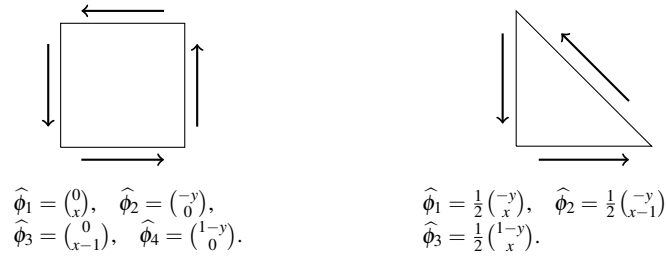

We again define one basis function $\psi_j$ of $Q_h$ for every element $T_k$ by $\psi_j|_{T_k}=\delta_{jk}$.
To any edge $e_i=T_l \cap T_r$, we now associate one single basis function $\phi_i$ defined by 
\begin{align} \label{eq:con}
\phi_i = \widetilde \phi_i + \widetilde \phi_{i+n_e}.
\end{align}
Using the construction of $\widetilde \phi_i$, one can give an equivalent definition of $\phi_i$ via 
\begin{equation} \label{eq:piola2}
\phi_i|_T = B_T^{-\top} \widehat \phi_{\alpha}, \qquad T \cap e_i \ne \emptyset,
\end{equation}
with basis functions $\widehat \phi_\alpha=\widehat \phi_{\alpha,0} + \widehat \phi_{\alpha,1}$ defined on the reference element in  Figure~\ref{fig:basis2}. 
Let us note that the space $V_h$ coincides with the Nedelec space of lowest order \cite{BoffiBrezziFortin13}.
Any function $\E_h \in V_h$ and $\H_h \in Q_h$ can then be expanded as 
\begin{align} \label{eq:exp}
\E_h = \sum\nolimits_i \e_i \phi_i \qquad \text{and} \qquad \H_h = \sum\nolimits_j \h_j \psi_j. 
\end{align}
As a consequence of \eqref{eq:con}, any $\E_h \in V_h$ can be interpreted as function $\widetilde \E_h \in \widetilde V_h$ by
\begin{equation}
\E_h 
= \sum\nolimits_i \e_i \phi_i 
= \sum\nolimits_i \e_i (\widetilde \phi_{i} + \widetilde \phi_{i+n_e}) 
= \sum\nolimits_i \e_i \widetilde \phi_i + \e_i \widetilde \phi_{i+n_e} 
=  \widetilde \E_h.
\end{equation}
The coordinates of $\widetilde \E_h$ and $\E_h$ are thus simply connected by $\widetilde \e_i =\widetilde \e_{i+n_e}=\e_i$.
Vice versa, we can associate to any function $\widetilde \E_h \in \widetilde V_h$ a function $\E_h = \Pi_h \widetilde \E_h \in V_h$ by defining its coordinates as $\e_i = \frac{1}{2}(\widetilde \e_i + \widetilde \e_{i+n_e})$. In linear algebra notation, this reads
\begin{equation}
\e = \P\,\widetilde \e  
\end{equation}
with projection matrix $\P$ defined by $\P_{ij}=\frac{1}{2}$ if $j=i$ or $j=i+n_e$, and $\P_{ij}=0$ else. 
%

We now define the system matrices for the system \eqref{eq:sys1}--\eqref{eq:sys2} by
$(\M_\mu)_{ij}=(\mu \psi_j,\psi_i)$, $\C_{ij}=\C'_{ji}=(\curl \phi_j,\psi_i)$, and
$\M_\epsilon^{-1} = \P\,\widetilde \M_\epsilon^{-1} \P^\top$,where $\widetilde \M_\epsilon$ is defined as in the previous sections. This construction has the folllowing properties.
\begin{lemma} \label{lem:3}
Let $\M_\mu$, $\C$, $\C'$, and $\M_\epsilon^{-1}$ be defined as above, and set $\M_\epsilon = (\M_\epsilon^{-1})^{-1}$.
Then the conditions (i)--(iii) are satisfied. 
\end{lemma}
\begin{proof}
Condition (i) follows by construction. 
The matrix $\M_\mu$ is diagonal and positive definite and therefore $\M_\mu^{-1}$ has the same properties. 
This verifies (ii) and (iii) for the matrix $\M_\mu$. 
Since $\P$ is sparse and has fully rank and $\widetilde \M_\epsilon^{-1}$ is block diagonal, symmetric, and  positive definite, one can see that also $\M_\epsilon^{-1}$ is sparse, symmetric, and positive-definite. 
This verifies conditions (ii) and (iii) for $\M_\epsilon$.
\qed
\end{proof}

In the following, we investigate more closely the relation of the system \eqref{eq:sys1}--\eqref{eq:sys2} with matrices as defined above and the system \eqref{eq:sys1w}--\eqref{eq:sys2w} discussed in the previous section.
We start with an auxiliary result. 
 \begin{lemma} \label{lem:4}
Let $\C$, $\P$, and $\widetilde \C$ be defined as above. 
Then one has $\widetilde \C =\C \P$. 
\end{lemma}
\begin{proof}
The result follows directly from the definition of the basis functions.
\qed
\end{proof}

As a direct consequence, we can reveal the following close connection between the methods \eqref{eq:sys1}--\eqref{eq:sys2} and \eqref{eq:sys1w}--\eqref{eq:sys2w} discussed in the preceding sections.
\begin{lemma} \label{lem:5}
Let $\widetilde \e(t)$, $\widetilde \h(t)$ be a solution of \eqref{eq:sys1w}--\eqref{eq:sys2w}. 
Then $\e(t) = \P\,\widetilde\e(t)$, $\h(t) = \widetilde \h(t)$ solves \eqref{eq:sys1}--\eqref{eq:sys2} 
with matrices $\M_\epsilon$, $\M_\mu$, and $\C$ as defined above.
\end{lemma}
\begin{proof}
From equation \eqref{eq:sys1w}, the definition of $\e$, $\h$, and Lemma~\ref{lem:4}, we deduce that 
\[
\dt \e 
= \P \dt \widetilde \e 
= \P \widetilde\M_\epsilon^{-1} \widetilde \C^\top \widetilde \h 
= \P \widetilde\M_\epsilon^{-1} \P^\top \C^\top \widetilde \h 
= \M_\epsilon^{-1} \C^\top \h.
\]
This verifies the validity of equation \eqref{eq:sys1}. 
Using equation \eqref{eq:sys2w},
we obtain
\[
\M_\mu \dt \h 
= \widetilde \M_\mu \dt \widetilde \h  
= -\widetilde \C\;\widetilde \e 
= -\C \P\;\widetilde \e = - \C\,\e,
\]
which verifies the validity of equation \eqref{eq:sys2}. Finally, using the discrete stability of the projection completes the proof. \qed
\end{proof}

\begin{remark}
The vectors $\e(t)$, $\h(t)$ computed via \eqref{eq:sys1}--\eqref{eq:sys2} with the above choice of matrices 
correspond to finite element approximations $\E_h(t) \in V_h$, $\H_h(t) \in Q_h$. 
Therefore, the procedure described above can be interpreted as a mixed finite element method with mass-lumping based on the approximation spaces $V_h$ and $Q_h$.
\end{remark}

As an immediate consequence of Lemma~\ref{lem:5} and the approximation result of Lemma~\ref{lem:2}, we now obtain the following assertions.
\begin{lemma} \label{lem:6}
Let $\e(t)$, $\h(t)$ denote the solutions of \eqref{eq:sys1}--\eqref{eq:sys2} with appropriate initial conditions 
and set $\E_h(t) = \sum_i \e_i(t) \phi_i$, $\H_h(t) = \sum_j \h_j(t) \psi_j$. 
Then 
\begin{align*}
\|\E_h(t) - \E(t)\| + \|\H_h(t) - \H(t)\| \le C h
\end{align*}
for all $0 < t \le T$. In addition, $\|\pi_h^0 \H(t) - \H_h(t)\| \le C h^2$ where $\pi_h^0 \H$ denotes the piecewise constant approximation of $\H$ on the mesh $\T_h$. 
\end{lemma}

By some elementary computations, one can verify the following observation.
\begin{lemma} \label{lem:7}
Let $\T_h$ be a uniform mesh consisting of orthogonal quadrilaterals $T$ of the same size.
Furthermore, let $\epsilon$ and $\mu$ be positive constants. 
Then the matrices $\M_\epsilon$, $\M_\mu$, and $\C$, defined above coincide with those 
obtained by the finite difference approximation on staggered grids; see \cite{Cohen02} for the two dimensional version. 
\end{lemma}

The method proposed in this section therefore can be understood as a variational generalization of the Yee scheme. 
In the two dimensional setting, one degree of freedom $\e_i$ is required for every edge, and one value $\h_j$ for every element.

\section{Numerical validation}\label{sec:4}
Consider the domain $\Omega = (-1,1)^2\setminus \{(x,y) : (x-0.6)^2+y^2 \le 0.25^2\}$,
which is split by an interior boundary into $\Omega = \Omega_1\cup\Omega_2$; see Figure~\ref{fig:2} for a sketch. We set $\epsilon=1$ on $\Omega_1$, $\epsilon=3$ on $\Omega_2$ and $\mu=1$ on $\Omega$, and consider a plane wave that enters the domain from the left boundary. The wave gets slowed down and refracted, when entering the domain $\Omega_2$, and reflected at the circle $\partial\Omega_0$, where we enforce a perfect electric boundary conditions. Convergence rates for the numerical solution are depicted in Table~\ref{tab:1} and a few snapshots of the solution are depicted Figure~\ref{fig:3}.
\begin{figure}[ht!]
\centering
\begin{minipage}{0.25\textwidth}\centering
\vskip-1em
\hspace*{-2em}
\begin{tikzpicture}[scale=1.2]
\draw[blue,line width=0.5mm] (-1,-1) -- (-1, 1); 
\draw[blue,line width=0.5mm] ( 1,-1) -- ( 1, 1);
\draw[blue, line width=0.5mm]  ( 1, 1) -- (-1, 1) node[above,midway] {\small $\partial\Omega_1$};
\draw[blue, line width=0.5mm]  (-1,-1) -- ( 1,-1);
\draw[blue, line width=0.5mm, dashed]  (-0.5,-1) -- ( 0, 1);
\draw[red,line width=0.5mm] (0.6, 0) circle (2.5mm) node[above,yshift=0.3cm] {\small $\partial\Omega_0$};;
\node[black] at (-0.5,0.5) {$\Omega_1$};
\node[black] at (0.1,-0.5) {$\Omega_2$};
\draw[] (-1,-1) node[below] {\tiny $(-1,-1)$};
\draw[] (1,-1) node[below] {\tiny $(1,-1)$};
\draw[] (-1,1) node[above] {\tiny $(-1,1)$};
\draw[] (1,1) node[above] {\tiny $(1,1)$};
\end{tikzpicture}
\vskip-1em
\caption{Geometry.\label{fig:2}}
\end{minipage}
\hfil
\begin{minipage}{0.6\textwidth}\centering
\centering
\vskip1em
\begin{tabular}{c||c|c|c||c|c} 
$h$ & DOF & $\tnorm \E_{h} - \pi_{h}\E_{h^*}\tnorm$ & eoc & $\tnorm \pi_h^0(\H_{h} - \pi_{h} \H_{h^*})\tnorm$ & eoc \\
\hline
\hline
\rule{0pt}{2.1ex}
$2^{-3}$ & 2246 & $0.158291$ & ---    &  $0.242490$ & ---    \\
$2^{-4}$ & 8884 & $0.057465$ & $1.46$ &  $0.069676$ & $1.80$ \\
$2^{-5}$ & 35368 & $0.025145$ & $1.19$ &  $0.017157$ & $2.02$ \\
$2^{-6}$ & 141136 & $0.011835$ & $1.08$ &  $0.004064$ & $2.07$
\end{tabular}
\caption{Errors and estimated order of convergence (eoc) with respect to a fine solution $(\E_{h^*},\H_{h^*})$ for $h^*=2^{-8}$. The total number of degrees of freedom (DOF) is also given.\label{tab:1}} 

\medskip
\end{minipage}
\end{figure}

\vspace{-3em}

\begin{figure}[ht!]
\begin{center}
\includegraphics[width=.17\textwidth]{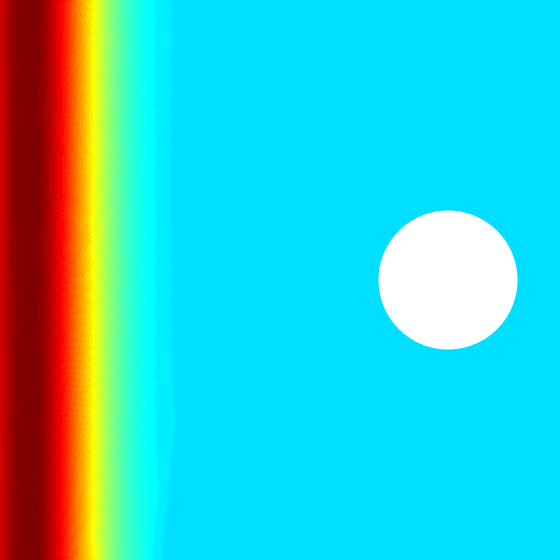}
\hskip2ex
\includegraphics[width=.17\textwidth]{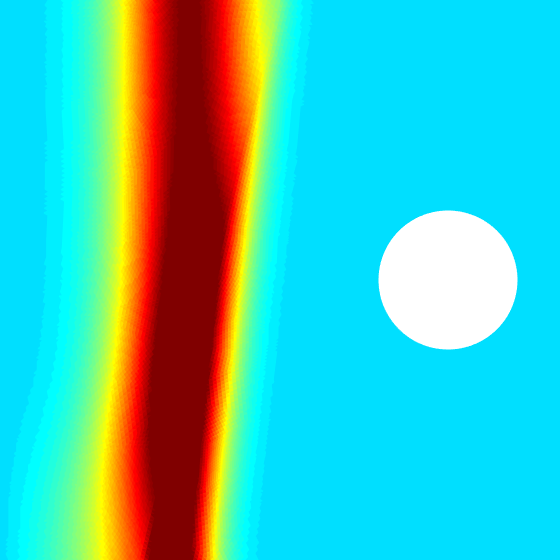} 
\hskip2ex
\includegraphics[width=.17\textwidth]{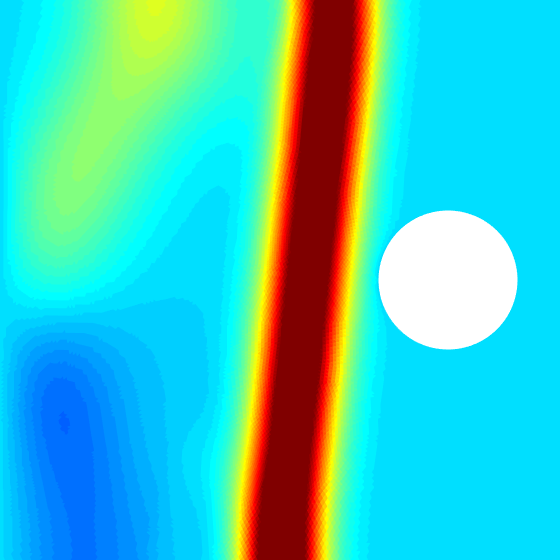}
\hskip2ex
\includegraphics[width=.17\textwidth]{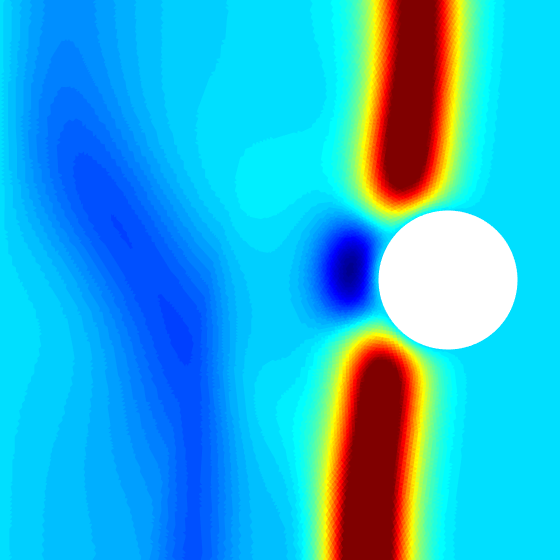}
\hskip2ex
\includegraphics[width=.17\textwidth]{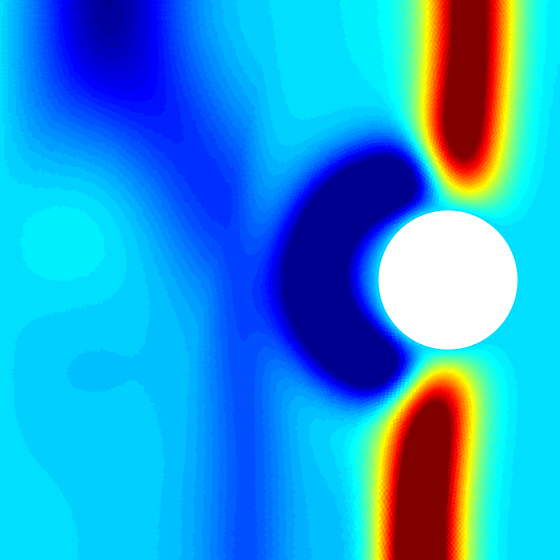} \\[1em]

\caption{Snapshots of the post-processed pressure fields $\widetilde p_h$ for time $t=0.8,1.2,1.6,2.4,2.8$ \label{fig:3}}
\end{center}
\end{figure}

\section{Discussion}\label{sec:5}

Before we conclude, let us briefly discuss an alternative formulation and the extension to three dimensions and higher order approximations. 

\begin{remark}
Eliminating $\h$ from \eqref{eq:sys1}--\eqref{eq:sys2}
leads to a second order equation 
\begin{align} \label{eq:sys3}
\M_\epsilon \dtt \e  + \K_{\mu^{-1}} \e = 0
\end{align}
for the electric field vector $\e$ alone, with $\K_{\mu^{-1}} = C' \M_{\mu}^{-1} \C$. 
A sufficient condition for the stability of the scheme \eqref{eq:sys3} is 
\begin{itemize}\setlength\itemindent{1em}
 \item[(iv)] $\quad \M_\epsilon$ and $\K_{\mu^{-1}}$ are symmetric and positive definite, respectively, semi-definite,
\end{itemize}
and for an efficient numerical integration of \eqref{eq:sys3}, one now requires that 
\begin{itemize}\setlength\itemindent{1em}
 \item[(v)]  $\quad \M_\epsilon^{-1}$ and $\K_{\mu^{-1}}$ can be applied efficiently.
\end{itemize}
The conditions (iv) and (v) can be seen to be a direct consequence of the conditions (i)--(iii), and the special form $\K_{\mu^{-1}} = \C' \M_\mu^{-1} \C$ of the matrix $\K_{\mu^{-1}}$. 
\end{remark}

\begin{remark}
Using the definition of the matrices $\M_\mu$, $\C$, and $\C'=\C^\top$ given in the previous section, 
one can verify that $\K_{\mu^{-1}}$ is given by $(\K_{\mu^{-1}})_{ij} = (\mu^{-1} \curl \phi_j, \curl \phi_i)$. 
Thus $\K_{\mu^{-1}}$ can be assembled without constructing $\C$ or $\M_{\mu}$ explicitly. Moreover, the conditions (iv) and (v) for $\K_{\mu^{-1}}$ are satisfied automatically.
The essential ingredient for a mass-lumped mixed finite element approximation of \eqref{eq:maxwell1}--\eqref{eq:maxwell2} thus is the construction of a positive definite and sparse matrix $\M_\epsilon^{-1}$. 
\end{remark}

\begin{remark}
The construction of the approximation $\M_\epsilon$ discussed in Section~\ref{sec:3} immediately generalizes to three space dimensions. Like in the two dimensional case, two basis functions $\widetilde \phi_i$, $\widetilde \phi_{i+n_e}$ of the space $\widetilde V_h$ are defined for every edge $e_i$ of the mesh and the approximation $(\cdot,\cdot)_h$ is defined via numerical quadrature by the vertex rule. The lumped mass matrix given by $(\widetilde \M_\epsilon)_{ij}=(\epsilon \widetilde \phi_j, \widetilde \phi_i)_h$ then is again block-diagonal. 
As before, the basis functions for the space $V_h$ are then defined by $\phi_i = \widetilde \phi_i + \widetilde \phi_{i+n_e}$ and the inverse mass matrix for the reduced space is again given by $\M_\epsilon^{-1} = \P\,\widetilde\M_\epsilon^{-1} \P^\top$ with projection matrix $\P$ of the same form as in two dimensions.
\end{remark}

\section*{Acknowledgements}
The authors are grateful for support by the German Research Foundation (DFG) via grants TRR~146, TRR~154, and Eg-331/1-1 and through grant GSC~233 of the ``Excellence Initiative'' of the German Federal and State Governments.

\end{document}